\documentclass{ws-jaa-ANNA}

\usepackage[urlcolor=blue]{hyperref} 

\usepackage{multirow} 
\usepackage{enumitem} 
\usepackage{xspace} 
\usepackage{listings} 

\usepackage{bm}
\usepackage{pdfsync}
\usepackage{subfigure}
\usepackage[dvipsnames]{xcolor}
\usepackage[english]{babel}

\usepackage[T1]{fontenc}

\renewcommand{\theenumi}{\alph{enumi}}  

\lstdefinelanguage{cocoa}
{
  commentstyle=\color{red!90!black},       
  stringstyle=\color{green!70!black},      
  morecomment=[l]{//},
  morecomment=[l]{--},
  morecomment=[s]{/*}{*/},
  morestring=[b]{"},
  classoffset=1,
  morekeywords={
    define,enddefine,if,endif,for,endfor,
    use,in,then,else,elif,return,
    and,or,
    break,continue,ciao,do,exit,
    ImportByValue,ImportByRef,importbyvalue,importbyref,
    in,isin,IsIn,
    on,opt,PrintLn,println,print,
    quit,ref,return,step,
    toplevel,TopLevel,
    then,to,
    use
  },
  keywordstyle=\color{blue!70!green!80!white},
  classoffset=2,
  morekeywords={
    Ideal,Mat,       Not,Record,Error,
    ideal,mat,matrix,not,record,error,submodule
  },
  keywordstyle=\color{purple!50!white},
  classoffset=3,
  morekeywords={
    TRUE,FALSE,True,False,true,false,
    Lex,Xel,DegLex,DegRevLex,
    PosTo,ToPos,Null
  },
  keywordstyle=\color{brown},
}

\newcommand{\ideal}[1]{\langle #1 \rangle}
\newcommand{\define}[1]{\textbf{\boldmath #1}}
\newcommand{\ext}{\;}

\newcommand{\BOX}[1]{%
  {
   \setlength{\fboxsep}{-0\fboxrule}
   \fbox{\hspace{1.2pt}\strut{$#1$}\hspace{1.2pt}}
   _{\mathstrut}
  }
}

\let\phi=\varphi
\let\rho=\varrho
\let\theta=\vartheta
\let\epsilon=\varepsilon

\def\LT{\mathop{\rm LT}\nolimits}
\def\MinLT{\mathop{\rm MinLT}\nolimits}
\def\LC{\mathop{\rm LC}\nolimits}
\def\LM{\mathop{\rm LM}\nolimits}
\def\NF{\mathop{\rm NF}\nolimits}

\def\den{\mathop{\rm den}\nolimits}

\def\lcm{\mathop{\rm lcm}\nolimits}

\def\notdiv{\!\mathbin{\not|}}
\def\divides{\mathbin{|}}
\def\SuchThat{\,\mathbin{|}\,}

\newcommand\degrevlex{\mathtt{DegRevLex}}
\newcommand\lex{\mathtt{Lex}}

\newcommand\Os{\mathrm{OrdMinLT}_\sigma}
\newcommand\Ot{\mathrm{OrdMinLT}_\tau}

\newcommand \ie {\textit{i.e.}}
\newcommand \eg {\textit{e.g.}}

\newcommand \rad{\mathop{\rm rad}\nolimits}

\newcommand \prim{\mathop{\rm prim}\nolimits}
\newcommand \FF {{\mathbb F}}
\newcommand \NN {{\mathbb N}}
\newcommand \QQ {{\mathbb Q}}

\newcommand \TT {{\mathbb T}}
\newcommand \ZZ {{\mathbb Z}}

\newcommand \tmin {{t'_{\min}}}
\newcommand \To{\longrightarrow}
\newcommand \TTo[1]{\mathop{\longrightarrow}\limits^{#1}}

\newcommand \FZ {{F_\ZZ}}

\newcommand \Gsigma {{G_{\sigma}}}
\newcommand \GsigmaZPrime {{G'_{\sigma,\ZZ}}}
\newcommand \Gtau {{G_{\tau}}}
\newcommand \GsigmaZ {{G_{\sigma,\ZZ}}}

\newcommand \ZZxn {\mathop{\ZZ[x_1, \dots, x_n]}}
\newcommand \QQxn {\mathop{\QQ[x_1, \dots, x_n]}}

\newcommand \longiso{\,\smash{\TTo{\lower 7pt\hbox{$\scriptstyle\sim$}}}\,}

\newcommand\groebner{Gr\"obner\xspace}
\newcommand\gbasis{Gr\"obner basis\xspace}
\newcommand\gbases{Gr\"obner bases\xspace}

\newcommand \cocoa{\mbox{\rm
C\kern-.13em o\kern-.07 em C\kern-.13em o\kern-.15em A}}
\newcommand \apcocoa{\mbox{\rm
A\kern-0.13em p\kern -0.07em C\kern-.13em o\kern-.07 em C\kern-.13em
o\kern-.15em A}}

\begin{document}

\markboth{J. Abbott, A.M. Bigatti, L. Robbiano}
{Ideals Modulo a Prime}

\catchline{}{}{}{}{}

\title{Ideals Modulo a Prime}

\author{\footnotesize John Abbott}

\address{ Dip. di Matematica,
\ Universit\`a degli Studi di Genova, \\ 
Via Dodecaneso 35,\
I-16146\ Genova, Italy\\
\email{abbott@dima.unige.it}}

\author{\footnotesize Anna Maria Bigatti}

\address{Dip. di Matematica,
\ Universit\`a degli Studi di Genova, \\ Via
Dodecaneso 35,\
I-16146\ Genova, Italy\\
\email{bigatti@dima.unige.it}}

\author{\footnotesize Lorenzo Robbiano}

\address{Dip. di Matematica,
\ Universit\`a degli Studi di Genova, \\ Via
Dodecaneso 35,\
I-16146\ Genova, Italy\\
\email{robbiano@dima.unige.it}}

\maketitle



\begin{abstract}
  The main focus of this paper is on the problem of relating an ideal
  $I$ in the polynomial ring $\QQ[x_1, \dots, x_n]$ to a corresponding
  ideal in $\mathbb F_p[x_1,\dots, x_n]$ where $p$ is a prime number;
  in other words, the \textit{reduction modulo $p$} of~$I$.  We 
  first define a new notion of $\sigma$-good prime for~$I$ which does depends
on the term ordering $\sigma$, 
but not on the given generators of~$I$.
We relate our notion of
  $\sigma$-good primes to some other similar notions already in the
  literature.  Then
we introduce and describe a new invariant called the universal denominator
which frees our definition of reduction modulo~$p$ from the term
ordering,
thus letting us show that all but finitely many primes  are good for~$I$.
One characteristic of our approach is that it enables us to easily detect
some bad primes, a distinct advantage when using modular methods.
\end{abstract}

\ccode{2010 Mathematics Subject Classification:
{13P25, 13P10, 13-04, 14Q10, 68W30}
}

\date{\today}



\section{Introduction and Notation}
\label{sec:Introduction and Notation}

There is a long tradition of using modular techniques for speeding up
computations which involve polynomials with rational coefficients.
Consequently, it is practically impossible to quote all the papers 
related to this topic; a few of them are~\cite{AoyNor},  \cite{IPS}, \cite{M04}, \cite{Nor2002},  \cite{NY[17]}, \cite{NY[18]}, \cite{NY2018},  \cite{Pa}, and~\cite{Wi}.
Two main interrelated obstacles to the success of this kind of approach are the
existence of \textit{bad, good and lucky primes}  and the difficulty of reconstructing the
\textit{correct rational coefficients} possibly in the presence of undetected bad primes.  We refer to~\cite{Abb2015} for a discussion
of the second problem and to~\cite{ABR} and~\cite{ABPR} for some new
results in this 
direction and applications to the problem of the implicitization of
hypersurfaces and of the computation of minimal polynomials.

The main focus of this paper is on the problem of relating an ideal $I$
in the polynomial ring $P=\QQ[x_1, \dots, x_n]$
to a corresponding ideal in $\mathbb F_p[x_1,\dots, x_n]$ where~$p$ is a prime number.
In other words, we face the problem of defining a \textit{reduction modulo~$p$} of~$I$.

To date there are two typical approaches to this problem.
One takes an arbitrary set of generators~$F \subseteq \QQ[x_1, \dots, x_n]$ of the 
ideal~$I$, then works with the ideal $\ideal{\prim(F)}$ generated by
the primitive parts $\prim(F) \subseteq \ZZ[x_1, \dots, x_n]$ where
each polynomial is scaled by a suitable rational so that its coefficients
become integers and with no common factor.
The reduction mod~$p$ of~$I$ is then defined to be the ideal in $\FF_p[x_1,\dots, x_n]$ generated by the reductions mod~$p$ of the elements in $\prim(F)$.
This approach has the merit of being easy to compute, but its main drawback is that it depends on the chosen generators of~$I$.

The other approach works with the projection of $I\cap  \ZZ[x_1, \dots, x_n]$ 
into the quotient ring $\FF_p[x_1,\dots, x_n]$.  This definition has the merit of being 
intrinsic to the ideal~$I$, but has the drawback of being not so easy to compute.
For a nice discussion about this topic, see for instance~\cite{NY2018}.

Our idea is different from both of these approaches.
We fix a term ordering~$\sigma$, and let  $\Gsigma \subseteq \QQ[x_1, \dots, x_n]$
be the reduced $\sigma$-Gr\"obner basis of $I$.
Solely for those primes~$p$ which do not divide the denominator of any
coefficient in $\Gsigma$, we
define the reduction mod~$p$ of~$I$ to be the ideal in $\FF_p[x_1,\dots, x_n]$ generated by the reductions mod~$p$ of the elements in $\prim(\Gsigma)$.

Our definition uses  the  ideal $\ideal{\prim(\Gsigma)} \subseteq \ZZxn$.
It turns out that  $\ideal{\prim(F)} \,\subseteq\, \ideal{\prim(\Gsigma)} \,\subseteq\, I\cap \ZZ[x_1, \dots, x_n]$.
It is interesting to observe that both inclusions can be strict. 
For instance, the first inclusion is strict
if we have $F = \{ 2x+y,\; y\}$; for then 
$\ideal{F} =\ideal{\prim(F)} =\ideal{2x+y,\;y} = \ideal{2x,y}$
while $\ideal{\Gsigma} = \ideal{\prim(\Gsigma)} = \ideal{x,\,y}$.
The second inclusion is strict
if  $I = \ideal{x- \tfrac{1}{2}z,\; y- \tfrac{1}{2}z} \subseteq \QQxn$, 
and~$\sigma$ is any term ordering with $x>_\sigma y>_\sigma z$;
for then $\Gsigma = \{x- \tfrac{1}{2}z,\; y- \tfrac{1}{2}z\}$, and
consequently,  $\ideal{\prim(\Gsigma)} =\ideal{2x-z,\; 2y-z}$
while $I\cap \ZZ[x_1,\dots, x_n] = \ideal{x-y,\; 2y-z}$.

Our definition has the merit of being independent of some arbitrary choice of system of generators of~$I$ and easily computable.  One possible objection is that it depends on the term ordering chosen.  But there is a nice way out, which uses the notion of the Gr\"obner fan of $I$,
and frees the definition of the reduction modulo~$p$ from the choice of~$\sigma$.

\smallskip
Here we give a more detailed description of the paper.
In Section~\ref{sec:Reductions modulo p} we use
results proved in~\cite{ABPR}, and introduce the notions of
\textit{$\sigma$-good} and \textit{$\sigma$-bad primes} for~$I$ with respect to a
given term ordering~$\sigma$,
which exploit the uniqueness of the reduced $\sigma$-\gbasis.
Notions of good and bad primes in modular computations are ubiquitous;
see for instance~\cite{BDFLP} for a fine discussion.  However, in our opinion
there is still room for improving the knowledge of this topic.
As a first result, we prove Theorem~\ref{thm:samereduction} which
relates the behaviour of good primes
with respect to two different term orderings.

From the theory of Gr\"obner Fans (see~\cite{MR88}) it follows
that for any ideal $I$ in $P$ all but finitely many primes
are good for all term orderings (see Remark~\ref{rem:deltone}).
  In other words there is an integer $\Delta$,
called the \textit{universal denominator} (see Definition~\ref{def:UniversalDen}),
such that for every prime $p$ which does not divide $\Delta$
we can define the reduction of $I$ to an ideal in $\FF_p[x_1,\ldots,x_n]$
which is independent of any term ordering (see
Definition~\ref{def:Ip}), and hence it depends only on~$I$.

In the context of polynomial ideals there are several notions of good and bad primes in the mathematical literature,
and Section~\ref{sec:Good primes vs lucky primes} is devoted to understanding
how they are interrelated.
We recall the notion of a \textit{minimal strong $\sigma$-\gbasis} for
ideals in $\ZZ[x_1, \dots, x_n]$
and, in Theorems~\ref{thm:orderedprimF}
and~\ref{thm:Rad=Rad}, we highlight the close relationship to
the reduced $\sigma$-\gbasis.
Following~\cite{Pa},
we say that $p$ is \textit{Pauer-lucky} for a set of polynomials 
${F \subseteq P}$, if it does not divide the leading coefficients of
any polynomial in a minimal strong $\sigma$-\gbasis
of $\ideal{\prim(F)}$; see Definition~\ref{def:IntegralPart}.
Then, given a term ordering~$\sigma$,
the ideal $I = \ideal{F}$,
and its reduced $\sigma$-\gbasis $\Gsigma$, we 
use the results contained in Theorems~\ref{thm:orderedprimF} 
and~\ref{thm:Rad=Rad} to
 show that if $p$ is Pauer-lucky for $\prim(F)$
then $p$ is $\sigma$-good for $I$ (see Proposition~\ref{prop:luckyandgood}), and that
$p$ is Pauer-lucky for $\prim(\Gsigma)$ if and only if it is $\sigma$-good for $I$
(see Corollary~\ref{cor:lucky=good}).

In Section~\ref{sec:DetectingBadPrimes} we address the problem of detecting $\sigma$-bad primes
when the reduced \text{$\sigma$-\gbasis} (in $\QQxn$) is not known.
In~\cite{Ar} E.A.~Arnold restricted her investigation to the case of
homogeneous ideals, and 
used suitable Hilbert functions to detect some bad primes.
We describe a similar but more general strategy.
The main new idea is to use the term ordering $\sigma$
to order tuples of power products.
In particular, we prove Proposition~\ref{prop:iprecgens} and the key
Lemma~\ref{lemma:lessthan} which pave the way for the proof of the main
Theorem~\ref{thm:sigma-tau} and its Corollary~\ref{cor:smallerisbad},
which gives a nice criterion for detecting relatively bad primes.
In essence, given two term
orderings~$\sigma$ and~$\tau$, and two primes $p$ and $q$ which are
both
\text{$\sigma$-good}, but only one is $\tau$-good, then we can determine which is $\tau$-good
just doing modular computations.

Apart from the theoretical advances already illustrated,
are there practical applications of the theoretical results proved in this paper? 
First experiments show that a modular approach for the computation of some \gbases
can benefit from our results. 

Most  examples described in the paper were computed using the
computer algebra system \cocoa\ (see~\cite{cocoalib} and~\cite{cocoa5}).
The computations of minimal strong \gbases were performed with
\textsc{Singular} (see~\cite{Singular}).

\subsection*{Notation}

For the basic notation and definitions about the theory of \gbases
see~\cite{KR1}, \cite{KR2}, and~\cite{KR3}.
The monoid  of power-products in~$n$ indeterminates
is denoted by $\TT^n$.
We use the convention that $\LT_\sigma(\ideal{0}) = \ideal{0}$.
In particular, if $t = x_1^{a_1}\cdots x_n^{a_n} \in \TT^n$ is a power-product
and $c$ is a coefficient,
we say that $t$ is a \define{term} and $c\, t$ is a \define{monomial}.
Throughout this article, when we use the notation $G = \{g_1, \dots, g_r\}$,
we actually mean that the $r$ elements in $G$ are numbered and distinct.
We use the symbol $\ZZ_\delta$  to represent the localization of $\ZZ$ at the
multiplicative system generated by the integer $\delta$.
Sometimes in the literature the symbol $\ZZ[\frac{1}{\delta}]$ is used instead of $\ZZ_\delta$.
If $p$ is a prime number, the symbol $\ZZ_{\ideal{p}}$ denotes the localization 
of $\ZZ$ at the maximal ideal $\ideal{p}$.


There are several instances in the paper where we compare
the minimal set of generators of two monomial ideals in different rings.
Hence  we introduce the following definition.
Let $K$ be a field, let $P=K[x_1,\dots, x_n]$ be a polynomial ring over $K$,
let $\sigma$ be a term ordering on $\TT^n$, and
let $I$ be an ideal in $P$. The unique minimal set of generators of $\LT_\sigma(I)$
is denoted by \define{$\MinLT_\sigma(I)$}$ \subseteq \TT^n$.
We observe that while $\LT_\sigma(I)$ is a monomial ideal in~$P$,
the set $\MinLT_\sigma(I)$ is a subset of $\TT^n$.
Later we introduce the tuple \define{$\Os(I)$} which contains the same elements as $\MinLT_\sigma(I)$
placed in increasing $\sigma$-order (see Definition~\ref{tupleofpowerproducts}).

Let $T = \{ t_1, t_2, \ldots, t_r \}$ be a set of power-products.  We
define the \define{interreduction} of~$T$ to be the unique maximal
subset $T'$ of~$T$ with the property  that there is no pair $(t_i, t_j)$
of distinct elements in $T'$ such that $t_i \divides t_j$.
We say that $T$ is \define{interreduced} if it is equal to its own
interreduction.

The \define{radical} of a positive integer $N$, \define{$\rad(N)$},
is the product of all primes dividing~$N$.
Obviously from the definition we
have $p \divides N \Longleftrightarrow p \divides \rad(N)$ for any prime $p$.
For example, $\rad(240) = 30$. 
Note that, for any positive integer $N$,
we have $\ZZ_N = \ZZ_\delta$ where $\delta = \rad(N)$.

Let $\delta$ be a positive integer, and~$p$ a prime number
not dividing $\delta$.  We write $\pi_p$ to denote
the canonical homomorphism $\ZZ_\delta \To \FF_p$ and all its natural
``coefficientwise'' extensions to
$\ZZ_\delta[x_1, \dots, x_n] \To \FF_p[x_1, \dots, x_n]$;
we call them all \define{reduction homomorphisms modulo~$p$}.

\section{Reductions modulo p}
\label{sec:Reductions modulo p}

In this section we analyse the concept of \textit{reduction modulo a prime $p$}.
In particular, we give a definition for the reduction mod $p$ of an ideal
which is independent of  the particular generators we have.

\goodbreak

\begin{definition}\label{def:densigma}
Let $P=\QQ[x_1,\dots,x_n]$.
\begin{enumerate}
\item Given a polynomial $f\in P$,
we define the \define{denominator of~$f$}, denoted by
\define{${\den}(f)$},
to be the positive least common multiple of
the denominators of the coefficients of~$f$.
In particular, we define $\den(0) = 1$.

\item Given a set of polynomials $F$ in $P$, we define the
   \define{denominator of~$F$}, denoted by \define{$\den(F)$},
to be the least common multiple of $\{ \den(f) \SuchThat f\in F \}$.
For completeness we define $\den(\emptyset) = 1$
where $\emptyset$ denotes the empty set.

\item Given a term ordering $\sigma$ and an ideal $I$ in $P$ with reduced
  $\sigma$-\gbasis $\Gsigma$, we define the
   \define{$\sigma$-denom\-inator of~$I$} to be
   \define{$\den_\sigma(I)$}$ = \den(\Gsigma)$.
\end{enumerate}
\end{definition}

The following easy example shows that $\den_\sigma(I)$
generally depends on $\sigma$.

\begin{example}\label{ex:DeltaDepends}
Let $P=\QQ[x,y]$ and let $g = x +2y\in P$, and let $I = \ideal{g}$.
Clearly $\{g\}$ is the reduced $\sigma$-\gbasis of $I$ with respect to
any term ordering $\sigma$ with $x >_{\sigma} y$.
Instead, the reduced $\tau$-\gbasis of $I$
with respect to any term ordering~$\tau$ with $y >_{\tau} x$
is $\{ y+\frac{1}{2}x\}$.
Therefore we have $\den_{\sigma}(I) = 1$ while $\den_{\tau}(I) = 2$.
\end{example}


The following lemma collects some important results taken
from~\cite{ABPR} (see also~\cite{NY[18]}).


\begin{lemma}\label{lemma:delta_f}
Let $P=\QQxn$, and let $\sigma$ be a term ordering on $\TT^n$.
Let $f \in P \setminus \{0\}$, and $I$ be an ideal in $P$ with reduced $\sigma$-\gbasis $\Gsigma$.
Furthermore, let~$\delta \in \NN^+$ be such that all coefficients
of $f$ and $G_{\sigma,f}$ are in $\ZZ_{\delta}$, where  $G_{\sigma,f}$ is the subset $\{g\in \Gsigma \mid \LT_\sigma(g)\le_\sigma \LT_\sigma(f)\}$.

\begin{enumerate}
\item  \label{item:rewriting}
Every intermediate step of rewriting $f$ via $\Gsigma$
has all coefficients in $\ZZ_{\delta}$.

\item The polynomial $\NF_{\sigma,I}(f)$ has all coefficients in $\ZZ_{\delta}$.
\end{enumerate}
\end{lemma}
\begin{proof}
Follows easily from~\cite[Lemma~3.2]{ABPR} restricted to $G_{\sigma,f}$.
\end{proof}

The following theorem is the foundation stone of our
investigation.  In particular, it sets the right context in which the
reduction mod $p$ of a \gbasis is the \gbasis of the ideal it
generates (claim~\ref{item:pi_p}).

\begin{theorem}\label{thm:FromMinPoly}\textbf{(Reduction modulo $p$ of Gr\"obner Bases)}\\
Let $P=\QQxn$, let $\sigma$ be a term ordering on $\TT^n$.
Let $I$ be an ideal in $P$ with reduced
$\sigma$-\gbasis $\Gsigma$.
Let $p$ be a prime number which does not divide $\den_\sigma(I)$.
\begin{enumerate}



\item \label{item:pi_p}
  The set $\pi_p(\Gsigma)$ is the reduced $\sigma$-\gbasis
of the ideal $\ideal{\pi_p(\Gsigma)}$.

\item The set of the residue classes of the elements in $\TT^n{\setminus}\LT_\sigma(I)$ is an $\FF_p$-basis
of the quotient ring $\FF_p[x_1, \dots, x_n]/\ideal{\pi_p(\Gsigma)}$.

\item For every polynomial $f\in P$ such that  $p \notdiv \den(f)$
we have the equality
 $\pi_p(\NF_{\sigma, I}(f)) = \NF_{\sigma, \ideal{\pi_p(\Gsigma)}}(\pi_p(f))$.


\end{enumerate}
\end{theorem}

\begin{proof}
See~\cite[Theorem~3.7]{ABPR}.
\end{proof}

\subsection{Good Primes}


Along the lines in~\cite{ABPR},
Theorem~\ref{thm:FromMinPoly} motivates the following definitions.

\begin{definition}\label{def:ugly}
Let $P = \QQ[x_1,\ldots,x_n]$.
\begin{enumerate}
\item Let $F$ be a finite set of polynomials in $P$.
We say that a prime $p$ is \define{bad} for $F$
if $p \divides \den(F)$, \ie~$p$ divides the denominator of at least one
coefficient of at least one polynomial in $F$.

\item Let  $\sigma$ be a term ordering on $\TT^n$,
let $I$ be an ideal in $P$, and let $\Gsigma$ be the
reduced $\sigma$-\gbasis of $I$.
If $p$ is bad for $\Gsigma$ we say that $p$ is \define{$\sigma$-bad for $I$}.
Otherwise we say that $p$ is  \define{$\sigma$-good for $I$}.

\item  If $p$ is a  $\sigma$-good prime for $I$
we define the \define{$(p,\sigma)$-reduction of $I$}
to be the ideal $I_{(p,\sigma)}$
 $=\ideal{\pi_p(\Gsigma)}\subseteq \FF_p[x_1,\dots,x_n]$
generated by the reductions modulo~$p$ of the polynomials in $\Gsigma$.

\end{enumerate}
\end{definition}

Now we can reinterpret
Theorem~\ref{thm:FromMinPoly}.\ref{item:pi_p} as follows.

\begin{remark}\label{rem:sameLT}
Let $P=\QQxn$,
and $\sigma$ a term ordering on $\TT^n$.
Let $I$ be an ideal in~$P$, and $\Gsigma$ its reduced
$\sigma$-\gbasis.  For every $\sigma$-good prime $p$ for $I$ we have
\begin{enumerate}
\item the set $\pi_p(\Gsigma)$ is the reduced $\sigma$-\gbasis
  of~$I_{(p,\sigma)}$, \ie~the ideal it generates.

\item 
$\MinLT_\sigma(I) = \MinLT_\sigma(I_{(p,\sigma)})$.
\end{enumerate}
\end{remark}

\begin{remark}
  We observe that the apparently simplistic definition, stating
  that $p$ is $\sigma$-good for $I$ if and only if $p$ does not divide
  $\den(\Gsigma)$, acquires a much deeper meaning after the above
  remark, and provides further support for the notation
  $I_{(p,\sigma)}$ since the reduced $\sigma$-\gbasis of any ideal is
  unique.
\end{remark}

Theorem~\ref{thm:FromMinPoly} turns out to be the essential tool for
proving the following result, which tells us that, for all but
finitely many primes, we may take simply $\ideal{\pi_p(F)}$ as the
reduction modulo $p$ of the ideal $\ideal{F}$.  Naturally the set of
suitable primes depends on $F$, the given system of generators.  This
dependence prompts us to prefer using reduced \gbases as generating
sets.


\goodbreak
\begin{theorem}\label{thm:GensAndRGB}
Let $\sigma$ be a term ordering on $\TT^n$, let $P=\QQxn$,
let $I$ be an ideal in~$P$, and let $\Gsigma$ be its reduced
$\sigma$-\gbasis.
Then let $F$ be any finite set of polynomials in the ideal~$I$, and let $\delta$  be
a positive integer such that both
$\Gsigma$ and $F$ are contained in $\ZZ_\delta[x_1, \dots, x_n]$.
Let~$p$ be a prime number such that $p \notdiv \delta$.
\begin{enumerate}
\item We have $\rad( \den_\sigma(I)) \divides \delta$.

\item We have
$\ideal{\pi_p(F)} \;\subseteq\; I_{(p,\sigma)}
\;\subseteq\; \FF_p[x_1, \dots, x_n]$.

\item
If there  exists a matrix $M$ with entries in~$\ZZ_\delta[x_1, \dots, x_n]$ such that
$\Gsigma = F\cdot M$, then we have $\ideal{\pi_p(F)} = I_{(p,\sigma)}$.
\end{enumerate}
\end{theorem}

\begin{proof} To prove claim~(a) we observe that the minimal localization of $\ZZ$
where $\Gsigma$ is contained is $\ZZ_{\den_\sigma(I)}[x_1, \dots, x_n]$, and the
conclusion follows.

To prove claim~(b) we observe that Theorem~\ref{thm:FromMinPoly}.\ref{item:pi_p} implies that
every element of~$F$ can be written as a linear combination of elements of $\Gsigma$
where the ``coefficients'' are polynomials in~$\ZZ_\delta[x_1, \dots, x_n]$.
In general, there will be several ways to reduce each element of~$F$ by the
basis $\Gsigma$, we may pick any one, and use the corresponding linear combination.
We can view $F$ and $\Gsigma$ as row-matrices by ordering their elements in some way.
Then writing the linear combinations as columns, we obtain a matrix~$A$ over
$\ZZ_\delta[x_1, \dots, x_n]$
(see Lemma~\ref{lemma:delta_f}.\ref{item:rewriting})
satisfying $F = \Gsigma\cdot A$.  This implies
that $\pi_p(F) = \pi_p(\Gsigma)\cdot \pi_p(A)$, concluding the proof since
$\pi_p(\Gsigma)$ generates $I_{(p,\sigma)}$.

Finally, we prove~(c). By claim~(a) the prime $p$ is $\sigma$-good for $I$,
hence we have the equality $I_{(p,\sigma)} = \ideal{\pi_p(\Gsigma)}$.
Moreover, we have $\pi_p(\Gsigma) = \pi_p(F)\cdot\, \pi_p(M)$ hence the implication
$I_{(p,\sigma)} \,\subseteq\, \ideal{\pi_p(F)}$ follows.
The reverse inclusion follows from~(b), and
the proof is complete.
\end{proof}

The following easy example shows that the inclusion in claim~(b) can be strict
even when $F$ is a generating set.

\begin{example}\label{ex:strictinclusion}
We follow the notation in the proof above.

\begin{lstlisting}[alsolanguage=cocoa]
/**/ use P ::= QQ[x,y,z], DegRevLex;
/**/ F := [x +2*z,  x +2*y];  I := ideal(F);
/**/ G := ReducedGBasis(I);  G;
[x +2*z,  y -z]
/**/ [GenRepr(g, I) | g in G];
[[1,  0],  [-1/2,  1/2]]
\end{lstlisting}

The ``new prime'' $2$ shows up in the denominators of the coefficients
representing the reduced $\sigma$-\gbasis elements as linear
combinations of the original generators.
Now we look at what happens modulo $2$ when we create an ideal from the
original generators, and when we create an ideal from the reduced \gbasis.

\begin{lstlisting}[alsolanguage=cocoa]
/**/ use P2 ::= ZZ/(2)[x,y,z], DegRevLex;
/**/ pi2 := PolyRingHom(P, P2, CanonicalHom(QQ,P2),indets(P2));
/**/ J2 := ideal(apply(pi2, F));
/**/ ReducedGBasis(J2);
[x]
/**/ I2 := ideal(apply(pi2, G));
/**/ ReducedGBasis(I2);
[y +z,  x]
\end{lstlisting}
\smallskip
\noindent
Here we see that the inclusion in Theorem~\ref{thm:GensAndRGB}.b
can be strict even though the prime $p=2$ is $\degrevlex$-good for $I$.
In the next section we shall see that $2$ is not a ``lucky prime'' for $F$.
\end{example}

Next we present the main result of this subsection.
It examines the situation when a prime is good with respect to
two different term orderings.

\begin{theorem}\label{thm:samereduction}
Let $\sigma$ and $\tau$ be two term orderings
on $\TT^n$, let $P=\QQxn$,  and let $I$ be an ideal in $P$.
Then let $\Gsigma$ and $\Gtau$ be the reduced \groebner
bases of~$I$ with respect to $\sigma$ and $\tau$, and let $p$ be a prime
which is both $\sigma$-good and $\tau$-good for~$I$.
\begin{enumerate}
\item We have the equality $I_{(p,\sigma)} = I_{(p,\tau)}$.

\item \label{item:tau}
The reduced $\tau$-\gbasis of $I_{(p,\sigma)}$ is $\pi_p(\Gtau)$.
\end{enumerate}
\end{theorem}

\begin{proof}
Since claim~(b) follows immediately from~(a) and Remark~\ref{rem:sameLT},
it is sufficient to prove claim~(a).
Let  $\delta = {\rm lcm}(\den_{\sigma}(I),\, \den_{\tau}(I))$, so
both $\Gsigma$ and~$\Gtau$ are contained in the ring $\ZZ_\delta[x_1, \dots, x_n]$.
From the assumption about $p$ we may apply Theorem~\ref{thm:GensAndRGB}
  with $F=\Gtau$    to deduce  that
 $I_{(p, \tau)} =\ideal{\pi_p(\Gtau)} \subseteq\, I_{(p, \sigma)}$.
  Applying Theorem~\ref{thm:GensAndRGB} again, after exchanging the roles
  of $\sigma$ and $\tau$, shows that
 $I_{(p, \sigma)} =\ideal{\pi_p(\Gsigma)} \subseteq I_{(p, \tau)}$.
  This proves the claim.
\end{proof}

\subsection{Universal Denominator}
In this subsection we recall some facts from  \groebner Fan Theory
(see~\cite{MR88}) and use them to define the \textit{universal denominator} of an ideal.

\begin{remark}\label{rem:deltone}
It is well-known that the \groebner fan of an ideal is finite (\eg~see~\cite{MR88}),
hence for every ideal in $\QQ[x_1,\ldots,x_n]$
there are only finitely many distinct reduced \groebner bases.
Each of these bases has its own corresponding denominator; thus any prime which
does not divide any of these denominators is good for all term orderings.
\end{remark}

This remark motivates the following definition.

\begin{definition}\label{def:UniversalDen}
Let $I$ be an ideal in $\QQ[x_1,\ldots,x_n]$.
Then the least common multiple of all~$\den_\sigma(I)$, as we vary $\sigma$,
is called the \define{universal denominator}
of $I$, and is denoted by~$\Delta(I)$.
\end{definition}

\begin{remark}
  We now see a big advantage of our choice of using reduced \gbases:
the finiteness of the ``fan'' of the reduced \gbases enables us to make this definition.
  If we allow more general generating sets then there is
  no finite ``universal denominator''.  For example, the ideal $\ideal{x,y}$
  admits generating sets such as $\{x+\tfrac{1}{p}y,\, y\}$ for any
  prime~$p$.
\end{remark}

Next we show the
existence of a well-behaved notion of reduction of $I$ modulo $p$
which is independent of the term orderings.

\begin{proposition}\label{prop:IdealModp}
Let $I$ be an ideal in $P$, let $\Delta(I)$ be its universal denominator,
and let $p$ be a prime not dividing $\Delta(I)$.
Then $I_{(p, \sigma)}$ does not depend on~$\sigma$.
\end{proposition}

\begin{proof}
For any term orderings $\sigma$ and $\tau$,
the prime $p$ is both $\sigma$-good and $\tau$-good.
So, by Theorem~\ref{thm:samereduction},
we have $I_{(p, \sigma)} = I_{(p, \tau)}$.
\end{proof}

This proposition motivates the following definition.

\begin{definition}\label{def:Ip}
  Let $I$ be a non-zero ideal in $P$, let $\Delta(I)$ be its universal
  denominator, and let $p$ be a prime not dividing $\Delta(I)$.  Then
  the \define{reduction of $I$ modulo $p$}, denoted~$I_p$, is the
  ideal $I_{(p, \sigma)}$, for any choice of~$\sigma$.
\end{definition}

 The main practical problem related to this definition is the computation of the
 universal denominator of $I$ which is, in general, not an easy task.
 Let us see some examples.

\begin{example}\label{ex:Deltone3}
Let $P=\QQ[x,y,z]$ and let
$I = \ideal{x^2  -y,\;  x y +z +1,\;  z^2  +x}$.
It is a zero-dimensional ideal and its  \groebner fan  consists of twelve cones.
\begin{lstlisting}[alsolanguage=cocoa]
/**/ use P ::= QQ[x,y,z];
I := ideal(x^2 -y,  x*y +z +1,  z^2 +x);
/**/ GF := GroebnerFanIdeals(I);
/**/ [ReducedGBasis(J) | J in GF];
\end{lstlisting}
\vskip-0.2cm
\begin{lstlisting}[alsolanguage=cocoa, basicstyle=\tiny\ttfamily]
  [z^2 +x,  x*y +z +1,  x^2 -y,  y^2 +x*z +x],
  [x +z^2,  y*z^2 -z -1,  z^4 -y,  y^2 -z^3 -z^2],
  [x*y +z +1,  z^2 +x,  x^2 -y,  x*z +y^2 +x,  y^3 +x -2*z -1,  y^2*z -y^2 -x -y],
  [x +z^2,  y*z^2 -z -1,  z^3 -y^2 +z^2,  y^2*z -y^2 +z^2 -y,  y^3 -z^2 -2*z -1],
  [z +x*y +1,  x^2 -y,  y^3 +2*x*y +x +1],
  [z +(-1/2)*y^3 +(-1/2)*x +1/2,  x^2 -y,  x*y +(1/2)*y^3 +(1/2)*x +1/2,
       y^5 +(-1/2)*y^4 +(1/4)*y^3 +(-7/4)*x -3*y^2 +(-5/2)*y +1/4],
  [z +(-2/7)*y^5 +(1/7)*y^4 +(-4/7)*y^3 +(6/7)*y^2 +(5/7)*y +3/7,
       x +(-4/7)*y^5 +(2/7)*y^4 +(-1/7)*y^3 +(12/7)*y^2 +(10/7)*y -1/7,  y^6 -2*y^3 -4*y^2 -y +1],
  [y -x^2,  z +x^3 +1,  x^6 +2*x^3 +x +1],
  [z^2 -y^3 +2*z +1,  x +y^3 -2*z -1,  y^2*z +y^3 -2*z -y^2 -y -1,  y^4 -2*y*z -z -y -1],
  [z^2 +2*z -y^3 +1,  x -2*z +y^3 -1,  y*z +(-1/2)*y^4 +(1/2)*z +(1/2)*y +1/2,
       y^5 +(-1/2)*y^4 +(-7/2)*z +2*y^3 -3*y^2 +(-5/2)*y -3/2],
  [y -x^2,  x^3 +z +1,  z^2 +x],
  [x +z^2,  y -z^4,  z^6 -z -1]
\end{lstlisting}
So we have $\Delta(I) = 2^2 {\cdot} 7$.
Consequently the reduction $I_p$  is
defined for every prime $p$ other than~$2$ and $7$,
and is generated by the reduction modulo $p$ of any of these \gbases.
\end{example}

\begin{example}\label{ex:BigUnivDenom}
While many ideals do have relatively small universal denominators,
a few seemingly simple ideals can have surprisingly large ones.  This
usually arises when the \groebner fan comprises many cones, which
can happen easily when there are many indeterminates.
We exhibit two examples with few indeterminates which nevertheless
have impressive denominators.

The first example in $\QQ[x,y,z]$ is the ideal $\ideal{x^2 y + x y^2
+ 1,\, y^3+x^2 z,\, z^3+x^2}$ whose universal denominator is larger
than $2 \times 10^{404}$ and has 105 distinct prime factors
(including all primes up to 100 except 79 and 89).
The \groebner fan of this ideal comprises 392 cones.

The second example is in the ring $\QQ[x,y,z,w]$: it is the apparently
innocuous ideal $\ideal{xyz+yzw+y,\, z^3+x^2,\, y^2z+w^3,\, x^3+y^3}$.
The \groebner fan of this ideal comprises almost 37000 cones, and
its universal denominator is larger than $2 \times 10^{379530}$.
This number has at least $24539$ distinct prime factors
including more than $\tfrac{2}{3}$ of all primes less than $2^{15}$;
in fact, the smallest prime not dividing the universal denominator is $4463$.
We checked the primeness of factors larger than $2^{32}$ using the function
\verb|mpz_probab_prime_p| from the GMP library, specifying $25$ iterations
of the Miller--Rabin test (see~ \cite{GMP-6.1.2}).

\end{example}

\section{Good primes vs lucky primes}
\label{sec:Good primes vs lucky primes}

In this section we recall some notions of \textit{lucky primes} which
have a long history, and compare them with our notion of \textit{good primes}.
We restrict our attention to the case where
the ring of coefficients is~$\ZZ$, although the theory is more general (see
for instance~\cite{AL} and~\cite{Pa}).
Several results described in this subsection are known, however
we adapt them to our notation, and for some of them we provide new proofs.

In this section we fix a term ordering $\sigma$ on the monoid $\TT^n$
of the power-products in~$n$ indeterminates, consequently
we sometimes omit the symbol $\sigma$.
Computations of minimal, strong \groebner bases were performed by
\textsc{Singular} (see~\cite{Singular}).

The first important tool  is the following definition (see \cite{AL},  Definition 4.5.6).

\begin{definition}\label{def:minimalstrongGB}
  Let $g_1, \dots, g_s$ be non-zero polynomials in
  $\ZZxn$.  We say that $\GsigmaZ = \{g_1, \dots, g_s\}$ is a \define{strong
    $\sigma$-\gbasis} for the ideal~$J = \ideal{\GsigmaZ}$, if for each $f \in
  J$ there exists some $i \in\{1,\dots, s\}$ such that
  $\LM_\sigma(g_i)$ divides~$\LM_\sigma(f)$.  We say that $\GsigmaZ$ is a
  \define{minimal strong $\sigma$-\gbasis} if it is a strong $\sigma$-\groebner
  basis and $\LM_\sigma(g_i)$ does not divide $\LM_\sigma(g_j)$ whenever
  $i\ne j$.
\end{definition}

\begin{remark}\label{rem:PID}
  In \cite{AL} the theory of minimal strong \groebner bases is fully developed,
  in particular it is stated that every non-zero ideal in $\ZZxn$
  has a minimal strong \gbasis (see~\cite{AL}, Exercise 4.5.9).

  It is well known that reduced \gbases have the property that the
  leading terms of their elements are pairwise distinct.  This also
  holds for minimal strong \gbases in $\ZZxn$ because the coefficient
  ring $\ZZ$ is a principal ideal domain.
\end{remark}

The following easy examples show the difference between
a minimal strong \gbasis of an ideal $J\subseteq\ZZxn$
 and the reduced \gbasis of the extended
ideal $J\ext\QQ[x_1,\ldots,x_n]$.  Note that,
whereas the elements of the reduced \gbasis are monic,
in a strong \gbasis the coefficients of the leading monomial
play an essential role in divisibility checking.

\begin{example}\label{ex:easyminstrongGB}
Let $J = \ideal{x^2,2x}$ be an ideal in $\ZZ[x]$.
Then $\GsigmaZ = \{x^2,\; 2x\}$ is a minimal strong $\sigma$-\gbasis
of $J$, while $\{x \}$ is the reduced $\sigma$-\gbasis of the
extended ideal $J\ext\QQ[x]$.

Let $\FZ=\{2x,\; 3y\} \subseteq \ZZ[x,y]$.  Then $\{2x,\; 3y,\; xy\}$ is a minimal strong $\sigma$-\gbasis
of the ideal~$\ideal{\FZ}$ for any term ordering $\sigma$, while $\{x,\, y \}$ is the reduced $\sigma$-\gbasis of the
extended ideal $\ideal{\FZ}\ext\QQ[x]$.
\end{example}

The reduced $\sigma$-\gbasis is a unique, canonical choice amongst all
{$\sigma$-\gbases}; in contrast, a minimal strong $\sigma$-\gbasis is not unique.

\begin{example}\label{ex:notuniqueminstrGB}
Let $\sigma$ be the $\degrevlex$ term ordering on $\TT^2$.
In the ring $\ZZ[x,y]$
let $\GsigmaZ = \{y^2 -x,\, 2x\}$ and $\GsigmaZPrime= \{y^2 +x,\, 2x \}$.
Then clearly $\ideal{\GsigmaZ} = \ideal{\GsigmaZPrime}$
and both $\GsigmaZ$ and $\GsigmaZPrime$ are
minimal strong $\sigma$-\groebner bases of this ideal.
The unique reduced $\sigma$-\gbasis of the extended ideal is $G = \{ x,\,y^2 \}$.
\end{example}


Although not unique, we shall now see that two minimal strong
$\sigma$-\groebner bases of an ideal $J$ in $\ZZxn$ share the same leading
monomials.

\begin{lemma}\label{lemma:twominimalstrongGB}
Let $J$ be an ideal in $\ZZxn$, and $\sigma$ be a term-ordering on $\TT^n$.
Let ${\GsigmaZ}$ and $\GsigmaZPrime$ be two minimal strong $\sigma$-\groebner bases of $J$.
Then
$\{\LM_\sigma(g) \mid g\in\GsigmaZ\} = \{\LM_\sigma(g') \mid g'\in \GsigmaZPrime\}$.
Consequently we have $\#\GsigmaZ = \#{\GsigmaZPrime}$ and
$\{\LC_\sigma(g) \mid g\in\GsigmaZ\} = \{\LC_\sigma(g') \mid g'\in \GsigmaZPrime \}$.
\end{lemma}
\begin{proof}
This equality can be proved
along the same lines as the proof of the uniqueness of the minimal
generating set of a monomial ideal in $K[x_1, \dots, x_n]$ where $K$
is a field~--~see for instance~\cite[Proposition~1.3.11.b]{KR1}.
\end{proof}

Given  a polynomial in $\QQxn$
we define its primitive integral part; it has
integer coefficients with no common factor, so
its modular reduction is non-zero for any prime~$p$.

\begin{definition}\label{def:IntegralPart}
Let $f$ be a non-zero polynomial in $\QQxn$, and let
$c$ be the integer content of $f {\cdot} \den(f) \in \ZZxn$.
Then the \define{primitive integral part of $f$},
denoted \define{$\prim(f)$},
is the primitive polynomial $c^{-1} f {\cdot} \den(f) \in \ZZxn$.
If $F$ is a set of non-zero polynomials in $\QQxn$
then \define{$\prim(F)$}$=\{\prim(f) \mid f {\in} F\}$.


For example, if $f = 2x + \tfrac{4}{3}$ then $\prim(f) = 3x+2$.
\end{definition}

\smallskip

Let $\Gsigma$ be the reduced $\sigma$-\gbasis of an ideal in $\QQxn$.  The
following theorem shows some important properties of all minimal
strong $\sigma$-\gbases of the ideal generated by $\prim(\Gsigma)$ in $\ZZxn$.

\begin{theorem}\label{thm:orderedprimF}
Let $P = \QQxn$, and $\sigma$ be a term-ordering on $\TT^n$.
Let $I$ be a non-zero ideal in $P$, and
let $\Gsigma=\{g_1,\ldots, g_r\}$ be its reduced $\sigma$-\gbasis,
whose elements are indexed so that $\LT_\sigma(g_1) <_\sigma \dots <_\sigma \LT_\sigma(g_r)$.
Then let $\GsigmaZ = \{ \tilde{g}_1,\ldots, \tilde{g}_s \}$ be a
minimal strong $\sigma$-\gbasis of the ideal~$J= \ideal{\prim(\Gsigma)} \subseteq \ZZxn$.
\begin{enumerate}
\item
  The elements in $\GsigmaZ$ can be indexed so that
  $\LT_\sigma(\tilde{g}_i) \!=\! \LT_\sigma(g_i)$ for $i \!=\! 1,\dots, r$
  while for $i \!=\! r{+}1,\dots,s$ each $\LT_\sigma(\tilde{g}_i)$ is a
  proper multiple of $\LT_\sigma(g_k)$ for some $k \le r$.

\item The subset $\{ \tilde{g}_1,\ldots, \tilde{g}_r \}$ is a minimal
  $\sigma$-\gbasis of $I$ in $P$.

\item We have 
$\LC_\sigma(\tilde{g}_i) \mid  \LC_\sigma(\prim(g_i))$ for $i=1,\dots, r$.

\item
If there exists a prime $p$ such that $p \!\mid\! \den(g_i)$ but
$p\!\notdiv\! \den(g_j)$ for every $j \!=\! 1,\dots,i{-}1$
then $p\mid\LC_\sigma(\tilde{g}_i)$.
\end{enumerate}
\end{theorem}

\begin{proof}
We start by proving claims~(a) and~(b).
For each $i=1,\dots, r$ we have $\prim(g_i) \in J$, hence
there is at least one polynomial $\tilde{g}_j \in \GsigmaZ$ such that
$\LM_\sigma(\tilde{g}_j) \mid \LM_\sigma(\prim(g_i))$.
Now~$\tilde{g}_j \in I$, hence
there is at least one polynomial $g_k \in \Gsigma$ such that
$\LT_\sigma(g_k) \mid \LT_\sigma(\tilde{g}_j)$.
Since $\Gsigma$ is a reduced \gbasis it follows that $k = i$, and then
also $\LT_\sigma(\tilde{g}_j) = \LT_\sigma(g_i)$.
So by suitably renumbering we may assume $j=i$.

Now we consider $i>r$.  Again we observe $\tilde{g}_i \in I$, hence
there is at least one polynomial $g_k \in \Gsigma$ such that
$\LT_\sigma(g_k) \mid \LT_\sigma(\tilde{g}_i)$.
Since $\GsigmaZ$ is minimal and $\LT_\sigma(\tilde{g}_k)=\LT_\sigma(g_k)$
we deduce from Remark~\ref{rem:PID} that $\LT_\sigma(\tilde{g}_i)$
 must be a proper multiple of $\LT_\sigma({g}_k)$.
We have now proved claims~(a) and~(b).

Next we prove claim~(c).
From claim~(a) it follows that the two polynomials
$\tilde{g}_i$ and $\prim(g_i)$
have the same leading term.
Since $\prim(g_i) \in J$
there is at least one polynomial $\tilde{g}_j \in \GsigmaZ$ such that
$\LM_\sigma(\tilde{g}_j) \mid \LM_\sigma(\prim(g_i))$.
This implies that $\LT_\sigma(\tilde{g}_j) \,| \LT_\sigma(g_i)$,
which in turn implies that $j=i$.  Hence
$\LC_\sigma(\tilde{g}_i) \,| \LC_\sigma(\prim(g_i))$.


Finally, we prove claim~(d).
Let $h = \tilde{g}_i - \LC_\sigma(\tilde{g}_i) {\cdot} g_i$, and observe that $h \in I$.
Using the fact that $\LM_\sigma(\tilde{g}_i) = \LC_\sigma(\tilde{g}_i)\,{\cdot}\LT_\sigma(g_i)$ we can write
$$
h 
  \;=\; \bigl(\tilde{g}_i - \LM_\sigma(\tilde{g}_i)\bigr) \,-\, \LC_\sigma(\tilde{g}_i){\cdot}\bigl(g_i  - \LT_\sigma(g_i)\bigr)
\;=\; \tilde{h}_i - \LC_\sigma(\tilde{g}_i){\cdot}h_i
$$
where
$\tilde{h}_i = \tilde{g}_i - \LM_\sigma(\tilde{g}_i)$ and
$       h_i =        g_i  - \LT_\sigma(g_i)$.
Now, since $h \in I$, we have
\[
0\;=\;\NF_{\sigma,I}(h) \;=\; \NF_{\sigma,I}( \tilde{h}_i) - \LC_\sigma(\tilde{g}_i)\cdot\NF_{\sigma,I}(h_i)
\]
Hence we have the equality
$\NF_{\sigma,I}( \tilde{h}_i) \,=\, \LC_\sigma(\tilde{g}_i)\,{\cdot}\NF_{\sigma,I}(h_i)$.
Given that $g_i \in \Gsigma$, the reduced $\sigma$-\gbasis of~$I$,
we have that $\NF_{\sigma,I}(h_i) = h_i$, which implies that
$\NF_{\sigma,I}( \tilde{h}_i) = \LC_\sigma(\tilde{g}_i) \cdot h_i$.

Now we look at the denominators of $\NF_{\sigma,I}(\tilde{h}_i)$
and $\LC_\sigma(\tilde{g}_i) \cdot h_i$.
Notice that $\tilde{h}_i$ has integer coefficients;
then using the fact that $\LT_\sigma(g_k) >_\sigma
\LT_\sigma(\tilde{h}_i)$ for all $k \ge i$
we can apply Lemma~\ref{lemma:delta_f} to conclude that
$\NF_{\sigma,I}(\tilde{h}_i) \in \ZZ_{\delta'}[x_1, \dots, x_n]$
where $\delta' = \lcm(\den(g_1),\dots,\den(g_{i-1}))$.
By hypothesis we know that $p \notdiv\, \delta'$,
thus $p \notdiv \den(\NF_{\sigma,I}(\tilde{h}_i))$.
Again, 
by hypothesis $p \!\mid\! \den(g_i)$.  Also, since $g_i$ is an element of a reduced \gbasis, it is monic; and the existence of $p$ implies it is not a monomial.
Thus we have $p \!\mid\! \den(h_i)$.
By the equality of the two normal forms we know that $p \notdiv \den(\LC_\sigma(\tilde{g}_i) \cdot h_i)$,
hence we necessarily have
$p \mid \LC_\sigma(\tilde{g}_i)$.
\end{proof}

The following example illustrates claim~(d).

\begin{example}
Let $\sigma = {\tt DegRevLex}$ on $\TT^2$ and
let $g_1 = y-\frac{1}{3},\, g_2 = x-\frac{1}{6} \in \QQ[x,y]$.
Then $\Gsigma = \{g_1, g_2\}$ is the reduced $\sigma$-\gbasis of $I =
\ideal{\Gsigma}$,
and $\GsigmaZ = \{3y-1,\, 2x-y,\,\allowbreak xy+y^2-x\}$
is a minimal strong $\sigma$-\gbasis of~$J = \ideal{\prim(\Gsigma)} \subseteq \ZZ[x,y]$
indexed according to claim~(a).
As stated in claim~(d):
\begin{itemize}
\item since $3 \mid \den(g_1)$, we therefore have $3 \mid
\LC_\sigma(\tilde{g}_1)$; indeed $\LC_\sigma(\tilde{g}_1)\,{=}\,3$.
\item since $2 \mid \den(g_2)$ and  $2 \notdiv \den(g_1)$,
we therefore have  $2 \mid \LC_\sigma(\tilde{g}_2)$; indeed $\LC_\sigma(\tilde{g}_2)\,{=}\,2$.
\end{itemize}
\end{example}

The following example illustrates the fact that simply sorting the
elements of a minimal strong \gbasis by increasing $\LT_\sigma$ may
not satisfy claim~(a).

\begin{example}
Let $P = \QQ[x,y,z]$ with term ordering $\sigma = {\tt DegRevLex}$ on $\TT^3$.
Let $g_1 = y-\frac{1}{3}$, $g_2=x-\frac{1}{2}$ and $g_3=z^3$.
Then $\Gsigma = \{g_1,g_2,g_3\}$ is the reduced \gbasis of the ideal $I = \ideal{G}$,
and we have
$\LT_\sigma(g_1) <_\sigma \LT_\sigma(g_2) <_\sigma \LT_\sigma(g_3)$.
A minimal strong \gbasis of the ideal~$J = \ideal{\prim(\Gsigma)} \subseteq \ZZ[x,y,z]$
with elements indexed according to claim~(a) is
$\GsigmaZ = \{3y-1,\, 2x-1,\, z^3, \, xy-x+y\}$,
but clearly we have $\LT_\sigma(\tilde{g}_3) >_\sigma \LT_\sigma(\tilde{g}_4)$.
\end{example}

Since the set of leading coefficients is independent of the specific
choice of minimal strong \gbasis of $J$, we make the following
definition.

\begin{definition}\label{def:lcm}
  Given a finite set $\FZ$ of non-zero polynomials in $\ZZxn$, we
  define
  \define{$\lcm_\sigma(\FZ)$}$ = \lcm\{\LC_\sigma(f) \mid f \in\FZ\}\in\ZZ$,
  the least common multiple of all the leading coefficients in $\FZ$.
  Given an ideal $J$ in $\ZZxn$ we define
  \define{$\lcm_\sigma(J)$}$= \lcm_\sigma(\GsigmaZ)$, where $\GsigmaZ$ is one of
  its minimal strong $\sigma$-\gbases.
\end{definition}

Now we apply Theorem~\ref{thm:orderedprimF} to show that the
primes appearing in $\den_\sigma(I)$
are the same as those appearing in the leading coefficients of any minimal
strong $\sigma$-\gbasis of the ideal generated by the primitive integral
parts of the reduced $\sigma$-\gbasis of~$I$.


\begin{theorem}\label{thm:Rad=Rad}
  Let $\sigma$ be a term ordering on $\TT^n\!,$ let $I$ be a non-zero
  ideal in $\QQxn$, and let $\Gsigma$ be its reduced
  $\sigma$-\gbasis.  Then
  $\rad(\den(\Gsigma)) = \rad(\lcm_\sigma(J))$ where $J=\ideal{\prim(\Gsigma)}$.
\end{theorem}

\begin{proof} Let 
$\GsigmaZ$ be a  minimal strong \gbasis of the ideal
$J \subseteq \ZZxn$.
The conclusion follows from the following two claims.\\
Claim (1): We have $\lcm_\sigma(J) \mid \den_\sigma(I)$ and hence
$\rad(\lcm_\sigma(J)) \mid \rad(\den_\sigma(I))$.\\
Claim (2): We have $\rad(\den_\sigma(I)) \mid \rad(\lcm_\sigma(J))$.

Let $\Gsigma = \{ g_1,\ldots,g_r\}$ indexed
so that $\LT_\sigma(g_1) <_\sigma \dots <_\sigma \LT_\sigma(g_r)$.
We shall also assume that $\GsigmaZ = \{ \tilde{g}_1,  \dots, \tilde{g}_s \}$
is indexed according to Theorem~\ref{thm:orderedprimF}.a.

Let us prove claim~(1).  From Theorem~\ref{thm:orderedprimF}.c
we get   $\LC_\sigma(\tilde{g}_i) \mid \LC_\sigma(\prim(g_i))$ for every~$i =1,\dots, r$.
Moreover, it is clear that $\LC_\sigma(\prim(g_i)) = \den(g_i)$, hence
we get $\LC_\sigma(\tilde{g}_i) \mid \den(g_i)$ for every~$i=1, \dots, r$.
Consequently, to finish the proof of claim~(1) we show that
$\lcm_\sigma(J) = \lcm_\sigma(\{\tilde{g}_1, \dots, \tilde{g}_r\})$.
Let $j$ be an index with $r+1 \le j \le s$; then by
Theorem~\ref{thm:orderedprimF}.a there exists
an index $i$ with $1 \le i \le r$ such
that $\LT_\sigma(\tilde{g}_i) \mid \LT_\sigma(\tilde{g}_j)$. Since $\GsigmaZ$
is a minimal strong \gbasis we have that
$\LC_\sigma(\tilde{g}_j) \!\mid\! \LC_\sigma(\tilde{g}_i)$.
Hence  $\lcm_\sigma(J) = \lcm_\sigma(\{\tilde{g}_1, \dots, \tilde{g}_r\})$.

Claim~(2) follows easily from Theorem~\ref{thm:orderedprimF}.d.
\end{proof}

The following example shows that in Theorem~\ref{thm:Rad=Rad}
it is not sufficient that $\Gsigma$ is just a \textit{minimal} $\sigma$-\gbasis
of $I$.

\begin{example}\label{ex:minimalnotsuff}
  Let $P = \QQ[x,y,z]$, and term ordering $\sigma = \degrevlex$ on $\TT^3$.
  Let $I = \ideal{yz-z^2,\; xy-z^2}$ be an ideal in $P$, then $\Gsigma =
  \{yz-z^2,\; xy-z^2,\; xz^2-z^3\}$ is its reduced $\sigma$-\gbasis.
  
  Let~$p$ be any prime.  The set $G_{\min} = \{yz-z^2,\; xy-z^2,\;
  xz^2-z^3+\frac{1}{p}(yz-z^2)\}$ is a minimal, but not reduced,
  $\sigma$-\gbasis of~$I$.  Clearly $\den(G_{\min}) = p$.  On
  the other hand, a minimal strong $\sigma$-\gbasis of the
  ideal $\ideal{\prim(G_{\min})}$ is $\GsigmaZ = \{yz-z^2,\; xy-z^2,\;
  xz^2-z^3\}$, hence $\lcm_\sigma(\ideal{\prim(G_{\min})})=1$.
\end{example}

The following example shows that under the assumptions of
Theorem~\ref{thm:Rad=Rad} we do not necessarily have the equality
$\den(\Gsigma) = \lcm_\sigma(\ideal{\prim(\Gsigma)})$.

\begin{example}\label{ex:RadNeeded}
Let $\sigma= \tt DegRevLex$,
let $I = \ideal{2x -y,\; 2y -z} \subseteq \QQ[x,y,z]$.
Its reduced $\sigma$-\gbasis is
$\Gsigma = \{y  -\tfrac{1}{2}z,\  x-\tfrac{1}{4}z\}$,
hence $\den(\Gsigma) = 4$.
A minimal strong \gbasis of the ideal $\ideal{\prim(\Gsigma)}$ is
$\GsigmaZ=\{ 2y -z,\  2x -y,\  xz -{y^2}^{\mathstrut}  \}$,
hence~$\lcm_\sigma(\GsigmaZ) =2$.
\end{example}

\begin{remark}
  We note that we can make claim~\ref{thm:orderedprimF}.d stronger:
  if $p$ is a prime satisfying the conditions in~\ref{thm:orderedprimF}.d
  then the greatest power of $p$ dividing $\den(g_i)$ is the same as the
  greatest power dividing $\LC_\sigma(\tilde{g}_i)$.  Observe that
  Example~\ref{ex:RadNeeded} does not contradict this stronger claim.
\end{remark}

\bigskip

Next we recall, using our setting and language, the definition of
lucky primes according to~\cite{Pa}.  Franz Pauer described
\textit{lucky ideals} (in~$R$) when the coefficient ring $R$ of the
polynomial ring is very general. Then he considered the case where $R$
is a principal ideal domain.  We rephrase his definition for the case
$R = \ZZ$.

\begin{definition}\label{def:PauerLucky}
  Let $\sigma$ be a term ordering on $\TT^n$, and
  let $\FZ \subseteq\ZZxn$ be a set of non-zero polynomials.
  Let $\GsigmaZ$ be a minimal strong $\sigma$-\gbasis of the ideal
  $\ideal{\FZ} \subseteq \ZZxn$.  A prime~$p$ is called
  \define{$\sigma$-Pauer-lucky for $\FZ$} (or simply
  \define{Pauer-lucky for $\FZ$} if $\sigma$ is clear from the
  context) if~$p$ does not divide the leading coefficient of any
  polynomial in $\GsigmaZ$.
\end{definition}

In~\cite[Proposition~6.1]{Pa}
Pauer proved the following relation between Pauer-lucky and
good primes.

\begin{proposition}\label{prop:luckyandgood}
  Let $\sigma$ be a term ordering on $\TT^n$, let $F
  \subseteq\QQxn$ be a set of non-zero polynomials, and
  let $p$ be a prime number.  If $p$ is Pauer-lucky for $\prim(F)$
  then~$p$ is $\sigma$-good for $\ideal{F}$.
\end{proposition}

The inclusion stated in this proposition can be strict,
as the following examples show.

\begin{example}\label{ex:strictinclusioncontinued2}
  Recalling Example~\ref{ex:strictinclusion} the prime~$2$ is
  good for the ideal $\ideal{F}$.  However, the minimal strong \groebner
  basis of the ideal $\ideal{\prim(F)}$ is $\{2y-2z,\; x+2z\}$, hence
  $2$ is not Pauer-lucky for $\prim(F)$.
\end{example}

\begin{example}\label{ex:GFandGZ}
Let $\sigma= \tt DegRevLex$,
and let
$F = \{x^2y - \tfrac{7}{2} y,\; xy^2 -\tfrac{3}{5}x\} \subseteq \QQ[x,y]$.
The reduced $\sigma$-\gbasis of the ideal $\ideal{F}$
is $\Gsigma = \{  xy^2 -\tfrac{3}{5}x,
\   x^2 -\tfrac{35}{6}y^2, \ y^3 -\tfrac{3}{5}y \}$.
Now we consider the two ideals
$\ideal{\prim(F)}$ and $\ideal{\prim(\Gsigma)}$ in $\ZZ[x,y]$.
A~minimal strong $\sigma$-\gbasis of $\ideal{\prim(\Gsigma)}$~is
$$\{ 6x^2-35y^2,\  5y^3-3y,\  5xy^2-3x,\  x^2y^2 -3x^2 +14y^2  \}$$
A~minimal strong $\sigma$-\gbasis of~$\ideal{\prim(F)}$~is
$$\{ 6x^2-35y^2, \  35y^3-21y,\   5xy^2-3x,
\  2x^2y-7y,\  x^2y^2-3x^2+14y^2   \}$$
Hence $\den_\sigma(\ideal{F}) =\den(\Gsigma)= \lcm_\sigma(\ideal{\prim(\Gsigma)}) = 2\cdot 3\cdot 5$, in
accordance with Theorem~\ref{thm:Rad=Rad},
while $\lcm_\sigma(\ideal{\prim(F)}) =  2\cdot 3\cdot 5\cdot 7$.
Consequently the prime $7$  is not Pauer-lucky for $\prim(F)$,
while  it is a good prime for the ideal $\ideal{F}$.
\end{example}

In view of the notion of Pauer-luckyness we can rephrase
Theorem~\ref{thm:Rad=Rad} as follows, which generalizes
the implication in Proposition~\ref{prop:luckyandgood}
(originally~\cite[Proposition~6.1]{Pa})
into an equivalence when $F$ is a reduced $\sigma$-\gbasis.

\begin{corollary}\label{cor:lucky=good}
Let $\sigma$ be a term ordering on $\TT^n$,
let $F \subseteq\QQxn$ be a set of non-zero polynomials,
let $\Gsigma$ be the  reduced $\sigma$-\gbasis of the ideal $\ideal{F}$.
Then a prime number $p$ is $\sigma$-Pauer-lucky for $\prim(\Gsigma)$ if and only
if it is $\sigma$-good for the ideal $\ideal{F}$.
\end{corollary}

We conclude the section by mentioning another important paper which deals with
a notion of lucky primes.

\begin{remark}\label{rem:Arnold}
In the paper~\cite{Ar}, Elisabeth Arnold considered the case where the polynomials in $F$
are homogeneous with respect to the standard grading,
and proves that, if~$\Gsigma$ is the reduced $\sigma$-\gbasis of $\ideal{F}$,
a prime $p$ is Pauer-lucky for $\prim(F)$ if and only
if the reduced \gbasis of $\ideal{\pi_p(\prim(F))}$ is $\pi_p(\Gsigma)$.
Moreover, this is also equivalent to $p$ being Hilbert-lucky and good for $\ideal{F}$.
\end{remark}

For a reformulation of this result and a nice example
see~\cite{BDFLP}, Theorem 5 and Example~6.

\section{Detecting Bad Primes}
\label{sec:DetectingBadPrimes}

With the fundamental help of Theorem~\ref{thm:samereduction}, we have
seen the nice relation between ideals generated by the reductions
modulo $p$ of two reduced \groebner bases of $I$ when~$p$ is good for both
term orderings.  But what happens when $p$ is good for one and bad for
the other?
We point out that the situation of knowing whether a prime
is good or bad for some particular term ordering does arise in some useful
circumstances: for instance, in implicitization where the
generators of the eliminating ideal (see~\cite{ABPR}) naturally form a
reduced \gbasis with respect to an elimination term ordering
for the \textit{dependent variables} (\ie~the coordinate indeterminates
to be used for expressing the implicit form).

\medskip

In the following we shall find it convenient to order finite sets of
distinct power-products.
For this reason we introduce the following definition.

\begin{definition}\label{tupleofpowerproducts}
Let $\sigma$ be a term ordering on~$\TT^n$ and let $P = K[x_1, \dots, x_n]$.
\begin{enumerate}
\item A tuple $(t_1, t_2, \dots, t_r)$  of  distinct power-products  in~$\TT^n$
is called \define{$\sigma$-ordered} if
we have $t_1<_\sigma t_2 <_\sigma \cdots <_\sigma t_r$.
The empty tuple is $\sigma$-ordered.

\item Let $F$ be a set or tuple of non-zero polynomials in $P$.
The $\sigma$-ordered tuple of the interreduction
of $\LT_\sigma(F)$ is denoted by   \define{$\Os(F)$}.

\item Let $I$ be an 
 ideal in $P$.  Then the $\sigma$-ordered tuple of
the leading terms of any minimal $\sigma$-\gbasis of $I$
is denoted by \define{$\Os(I)$}.  In particular, if $I$ is the zero ideal
then $\Os(I)$ is the empty tuple.
\end{enumerate}
\end{definition}

\begin{example}
Let $P = \QQ[x,y]$ and let $\sigma = \degrevlex$.
We consider the set of polynomials $F=\{x{+}y{+}1,\, x^2{+}2x{+}y{+}1,\, y^3\}$.
Observe that $\LT_\sigma(F) = \{x, x^2, y^3\}$ is not interreduced;
interreduction produces $\Os(F) = (x, y^3)$.
In contrast, working with the ideal $I = \ideal{F}$ gives $\Os(I) = (y,\,  x)$
since the reduced \gbasis is $\{x{+}1,\, y\}$.
\end{example}

%
%
%
%
%

We define a total ordering on the $\sigma$-ordered tuples
of distinct power-products.

\begin{definition}\label{def:orderedcomparison}
Let $\sigma$ be a term ordering on the monoid~$\TT^n$,
and let $T = (t_1, \ldots, t_r )$ and
$T' = (t'_1, \ldots, t'_{r'} )$ be $\sigma$-ordered tuples
of distinct power-products in $\TT^n$.
We say that  \define{$T'\ \sigma$-precedes~$T$} and write
 \define{$T'\prec_\sigma T$} if
either $T$ is a proper prefix of $T'$, \ie~$r<r'$
and $t_i = t'_i$ for all $i=1,\ldots,r$,
or there exists an index $k \in \{1,\dots, \min(r, r') \}$
such that $t_i = t'_i$ for every $i=1,\dots, k{-}1$ and
$t'_k <_\sigma t_k$.

We write \define{$T'\preceq_\sigma T$} to mean either
$T'\prec_\sigma T$ or $T'=T$.
\end{definition}

\begin{remark}\label{rem:partialtotal}
  We observe that ``$\sigma$-precedes'' is just the
  ``$\sigma$-lexicographical'' ordering
  on the $\sigma$-ordered tuples $(T, x^\infty )$ where $T$ is a
  $\sigma$-ordered tuple of distinct power-products, and
  $x^\infty$ is $\sigma$-greater than any power-product.
  For instance, we now easily see that every non-empty tuple $\sigma$-precedes the empty tuple.
\end{remark}

\begin{example}
  \label{ex:sigma-precedes}
Let $\sigma = \lex$ on $\TT^3$ with $x >_\sigma y >_\sigma z$.  We compare these tuples:\\
$
\begin{array}{ll}
(z,y,x) \prec_\sigma (z,y) \qquad &
\text{--- since } x <_\sigma x^\infty, \;\text{equivalently, } (z,y) \text{ is a proper prefix}\\
(z,y) \prec_\sigma (z,y^2,x) &
\text{--- since } y <_\sigma y^2\\
\end{array}
$
\end{example}

\begin{proposition}\label{prop:iprecgens}
Let $P= K[x_1, \dots, x_n]$, and $\sigma$ be
a term ordering on $\TT^n$.
Let $J$ be an ideal in~$P$, and
let $F$ be a set of non-zero polynomials in~$J$.
\begin{enumerate}
\item $\Os(J)=\Os(F)$ if and only if $F$ is a $\sigma$-\gbasis
  of~$J$.

\item $\Os(J)\prec_\sigma\Os(F)$ if $F$ is not a $\sigma$-\gbasis of~$J$.
\end{enumerate}

\end{proposition}

\begin{proof}
By definition, $F\subseteq J$ is  a $\sigma$-\groebner
basis of $J$ if and only if $\LT_\sigma(F)$ generates~$\LT_\sigma(J)$.
Hence claim~(a) follows.
Now we prove claim~(b).

Since $F$ is not a $\sigma$-\gbasis of $J$, we have $\Os(F) \neq \Os(J)$.
If it happens that $\Os(F)$ is a proper prefix of $\Os(J)$, the conclusion follows immediately.
So we assume that $\Os(F)$ is not a proper prefix.
Note that $\Os(J)$ cannot be a proper prefix of $\Os(F)$ as otherwise this would imply
that there is~$f \in F \subseteq J$ with $\LT_\sigma(f) \notin \LT_\sigma(J)$.

Let $\Os(F) = (t_1, t_2, \dots)$,  let $\Os (J) = (t'_1, t'_2, \dots)$,
and let $k$ be the first index such that $t_k\ne t'_k$.
Since $F \subseteq J$ we know that $t_k \in \LT_\sigma(J)$,
and hence $t_k$ is a multiple of some element of $\Os(J)$.
Since $\Os(F)$ is interreduced, $t_k$ is not a multiple of any of
the other $t_j$, and thus specifically not a multiple of any of $t'_1,\ldots,t'_{k-1}$.
Hence $t_k$ can only be a non-trivial multiple of $t'_k$ or a multiple
of $t'_j$ for some index $j > k$.  Either way $t_k >_\sigma t'_k$,
and so $\Os(J) \prec_\sigma \Os(F)$ as claimed.
\end{proof}

The next example illustrates the importance of
$\Os(F)$ being interreduced.

\begin{example}
Let $P = K[x,y]$ and let $\sigma {=} \degrevlex$.
Let $J = \ideal{x, y^3}$ and consider the $\sigma$-ordered
tuple $T = (x, x^2, y^3)$; the elements of $T$ are clearly non-zero
polynomials in $J$.  We observe that
$\Os(J) = \Os(T) = (x, y^3)$.
However, the tuple $T$ is not interreduced, and we have $T \prec_\sigma \Os(J)$.
\end{example}

We recall here a standard result from the theory of \gbases; for
the sake of completeness we include the proof.

\begin{lemma}\label{lemma:LTIsubsetLTJ}
Let $P= K[x_1, \dots, x_n]$, let $\sigma$ be a term ordering on $\TT^n$, and
let $I$, $J$ be ideals in $P$.
If $I \subsetneq J$ then $\LT_\sigma(I) \subsetneq \LT_\sigma(J)$.
\end{lemma}
\begin{proof}
Since $I \subsetneq J$ we clearly have
$\LT_\sigma(I) \subseteq \LT_\sigma(J)$.  Let $f \in J \setminus I$
with minimal \hbox{$\sigma$-leading} term, thus
$\LT_\sigma(f) \in \LT_\sigma(J)$.  However, by the minimality of
$\LT_\sigma(f)$ we see that $f$ cannot be head-reduced
by any element of a
$\sigma$-\gbasis of $I$.  Hence we conclude that
$\LT_\sigma(f) \not\in \LT_\sigma(I)$.
\end{proof}

We are ready to prove the following interesting result.

\begin{corollary}\label{cor:inverseprec}
Let $P= K[x_1, \dots, x_n]$, let $\sigma$ be a term ordering on $\TT^n$, and
let $I$, $J$ be
ideals in $P$.
If $I\subsetneq J$ then $\Os(J) \prec_\sigma \Os(I)$.
\end{corollary}

\begin{proof}
Let $\Gsigma$ be a
$\sigma$-\gbasis of $I$.
Thus $\Gsigma$ is a  set of non-zero polynomials in~$J$.
By Proposition~\ref{prop:iprecgens}
we have $\Os(J)\preceq_\sigma\Os(\Gsigma) = \Os(I)$.
From Lemma~\ref{lemma:LTIsubsetLTJ} and the assumption that $I\subsetneq J$
the conclusion follows.
\end{proof}

Next we prove another useful result.

\begin{lemma}\label{lemma:lessthan} 
Let $\sigma$ be a term ordering on $\TT^n$.
Let $T = (t_1, t_2, \dots, t_r)$ be an interreduced $\sigma$-ordered
tuple of elements in $\TT^n$, and let
$T'$ be another set of elements in~$\TT^n$.
Assume that there exist $t' \in T'$ and an index $k$ 
such that:
\[ 
\bullet\; t_1,\ldots,t_{k-1} \in T' \qquad
\bullet\; t_k >_\sigma t' \qquad
\bullet\; t' \text{ is not divisible by any }t_i\in T
\]
Then $\Os(T') \prec_\sigma T$, and $T$ is not a proper prefix of $\Os(T')$.
\end{lemma}

\begin{proof}
Let $\tmin = \min_\sigma \{ \tilde{t}\in T'
 \SuchThat \tilde{t} \text{ not divisible  by any } t_i\in T \}$,
and let $j$ be the smallest index such that  $t_j >_\sigma \tmin  $.  
These definitions imply that $\tmin  >_\sigma t_{j-1}$, and also $j \le k$, so we
know that $t_1,\ldots,t_{j-1} \in T'$.

Now we define the tuple $\overline{T'} =
(t_1,\dots,t_{j-1}, \tmin)$, which is clearly $\sigma$-ordered.  
Furthermore, we have that $\overline{T'} \prec_\sigma T$ because
 $t_j >_\sigma \tmin$ by construction.

Next we prove that 
$\overline{T'}$ is a prefix of $\Os(T')$,
and therefore satisfies
$\Os(T') \preceq_\sigma\overline{T'} \prec_\sigma T$.

The set of power-products in $\overline{T'}$ is interreduced: we
already know that $\{ t_1, \dots, t_{j-1}\}$ is interreduced, and
$\tmin$ is not divisible by any of them; on the other hand, we see
that~$\tmin$ cannot divide any of them because it is the
$\sigma$-greatest element.

Now, it suffices to show that each element of $T'$ (or, equivalently,
of $T' \setminus \overline{T'}$)
is either $>_\sigma\tmin$
or a multiple of an element of $\overline{T'}$.
Let $s' \in T'$; we shall argue depending on whether $s'$ is divisible by some element of the tuple $T$.
First we consider the case where $s'$ is not divisible by any $t_i \in T$.
By definition of $\tmin$ we see that $s' \ge_\sigma \tmin$;
if $s' = \tmin$ it is trivially a multiple of an element of $\overline{T'}$,
otherwise $s' >_\sigma \tmin$ as claimed.
We address now the case where $s'$ is a multiple of some $t_i \in T$.
If $i < j$ then~$s'$ is clearly a multiple of an element of $\overline{T'}$;
otherwise $i \ge j$, so $s' \ge_\sigma t_i \ge_\sigma t_j >_\sigma \tmin$.

In conclusion, $\Os(T') \prec_\sigma T$,
and $T$, not containing $\tmin$, is not a proper prefix of $\Os(T')$.
\end{proof}

The following example illustrates the steps in this proof.

\begin{example}\label{ex:proof}
Let $\sigma = \degrevlex$.
Consider the interreduced $\sigma$-ordered tuple $T$
and the set $T'$:
\begin{center}
$T =       (xyz, x^3, \BOX{x^2y^2}, xz^4, y^6, z^7)$
\\
$T' = \{xyz, x^3, x^2z^2, \BOX{xy^2}, y^7, x^2y^8\}$.
\end{center}
We take $k = 3$, so $t_k = x^2y^2$, and $t'=xy^2$,
which is not a multiple of any power-product in $T$:
these choices satisfy the hypotheses of the lemma.
Following through the proof we have $\tmin=t'$, $j=2$ and $\overline{T'} = (xyz,
\BOX{xy^2})$, and we see clearly that $\overline{T'}\prec_\sigma T$.
We compute the tuple $\Os(T') = (xyz, \BOX{xy^2}, x^3, x^2z^2, y^7)$,
and observe that $\overline{T'}$ appears as a (proper) prefix.
Consequently $\Os(T')\prec_\sigma \overline{T'}\prec_\sigma T$.
\end{example}

The following easy example shows the importance of
the non-divisibility assumption in the lemma.

\begin{example}\label{ex:bessential}
Let $\sigma = \degrevlex$ and let $T = (x, y^3)$, an interreduced $\sigma$-ordered tuple.
Now let
$T' = \{x, x^2, y^4, z^4\}$.
For $k=1$
there is no $t' \in T'$ with $t' <_\sigma t_k$;
and for $k=2$ the only elements of $T'$ which are $\sigma$-less-than $t_k$ are
$x$ and $x^2$, but both are divisible by $t_1$.  So we cannot apply the lemma.
Indeed, $\Os(T') = (x, y^4, z^4)$, and we have $T \prec_\sigma \Os(T')$.
\end{example}

Now we are ready to prove the main theorem of this section.

\begin{theorem}\label{thm:sigma-tau}
Let $P=\QQxn$, let $\sigma$, $\tau$ be two term orderings
on $\TT^n$, let~$I$ be a non-zero ideal in $P$, and
let $p$ be a prime which is  $\sigma$-good for $I$.
\begin{enumerate}
\item If $p$ is $\tau$-good for $I$,
we have ${\Ot( I_{(p,\,\sigma)})}^{\mathstrut}= \Ot(I)$.

\item If $p$ is $\tau$-bad for $I$,
we have $\Ot( I_{(p,\,\sigma)}) \prec_\tau \Ot(I)$,
and also that $\Ot(I)$ is not a proper prefix of $\Ot( I_{(p,\,\sigma)}) $.
\end{enumerate}
\end{theorem}

\begin{proof}
Let $\Gtau$ be the reduced $\tau$-\gbasis of $I$, and $\Gsigma$
be the reduced $\sigma$-\gbasis for $I$.

We start by proving claim~(a).
By hypothesis, $p$ is both $\sigma$-good and $\tau$-good for~$I$,
hence Theorem~\ref{thm:samereduction}.\ref{item:tau} implies that the reduced
$\tau$-\groebner  basis of $I_{(p,\,\sigma)}$ is $\pi_p(\Gtau)$ which
has the same leading terms as $\Gtau$, and the conclusion follows immediately
from Remark~\ref{rem:sameLT}.b.

Now we prove claim~(b).
Let $\Gtau =\{g_1,\dots,g_r\}$ where the elements are indexed so that
$\LT_\tau(g_i) <_\tau \LT_\tau(g_{i+1})$ for $i = 1, \dots,r{-}1$.
For $i = 1, \dots,r$,
let $\tilde{g}_i = \prim(g_i)$ so in particular
$\pi_p(\tilde{g}_i)\ne 0$.
Define the following $\tau$-ordered tuple $T$ and set $T'$
\begin{center}
$
\begin{array}{cccccccccc}
T &= &(&\LT_\tau(g_1)&,&\dots&,& \LT_\tau(g_r)&) &\text{ which is just }\Ot(I)\\
T'&= &\{&\LT_\tau(\pi_p({\tilde{g}_1}))&,&\dots&, &
\LT_\tau(\pi_p({\tilde{g}_r}))&\}\\
\end{array}
$
\end{center}
By definition of $\pi_p$ we have $\LT_\tau(\pi_p(\tilde{g}_i))  \le_\tau \LT_\tau(g_i)$ for all $i$.
Since $p$ is $\tau$-bad, there is at least one index $j$ such that
$p$ divides the
denominator of $g_j \in \Gtau$,
hence it divides also the leading coefficient of ${\tilde{g}_j}$.
Therefore
$\LT_\tau(\pi_p({\tilde{g}_j}))  <_\tau \LT_\tau(g_j)$;
let $k$ be the smallest such index.
Moreover, since $\Gtau$ is a reduced \gbasis,
$\LT_\tau(\pi_p({\tilde{g}_k}))$ is not a 
multiple of any element of $\LT_\tau(\Gtau)$.
Therefore we can apply Lemma~\ref{lemma:lessthan} to $T$ and $T'$
with the above value of $k$
and deduce that $\Ot(T')\prec_\tau T = \Ot(I)$, and that $\Ot(I)$
is not a proper prefix of $\Ot(T')$.

Let $F = \{\pi_p({\tilde{g}_1}),\dots, \pi_p({\tilde{g}_r})\}$.
By Lemma~\ref{lemma:delta_f}.b we deduce that
$F\subseteq I_{(p,\,\sigma)}$.  Hence
Proposition~\ref{prop:iprecgens} implies that
$\Ot(I_{(p,\,\sigma)}) \preceq_\tau \Ot(F) = \Ot(T')$.

Combining the two inequalities, the conclusion follows.
\end{proof}

This theorem enables us to
detect some bad primes without having to compute the
reduced $\tau$-\gbasis of $I$
over the rationals.

\begin{corollary}\label{cor:smallerisbad}
Let $P=\QQxn$, let $\sigma$ and $\tau$ be two term orderings
on $\TT^n$, and let~$I$ be a non-zero ideal in $P$.
Let $p$ and $q$ be $\sigma$-good primes for $I$.\\
If $\Ot( I_{(q,\,\sigma)}) \prec_\tau {\Ot(I_{(p,\,\sigma)})}$
then $q$ is $\tau$-bad for $I$.
\end{corollary}

\begin{proof}
By Theorem~\ref{thm:sigma-tau} we know that
$\Ot( I_{(p,\,\sigma)}) \preceq_\tau \Ot(I)$.
Hence $\Ot( I_{(q,\,\sigma)}) \prec_\tau {\Ot(I)}$,
so Theorem~\ref{thm:sigma-tau}(a) implies that the prime $q$ is
$\tau$-bad for~$I$.
\end{proof}

\begin{example}\label{ex:manybadprimes}
In the polynomial ring $\QQ[x,y,z]$ with term ordering $\sigma = \degrevlex$,
let $F = \{x^2 y +7xy^2 -2,\;  y^3 +x^2z,\;  z^3 +x^2 -y\}$
and let $I = \ideal{F}$. It turns out that all primes are $\sigma$-good for $I$,
and we have
$$\Os(I) =   \Os(I_{(p,\sigma)})  = (z^3, \; y^3,\;  x^2y,\;  x^4z, \; x^6)
\hbox{\quad for all primes\ } p.$$
Now we consider the term ordering $\tau = \tt Lex$. It turns out that
the set of $\tau$-bad primes for $I$
is $S=\{2, \;  7, \;  11, \;  55817,\; p\}$ where $p \,\approx 1.8 \times 10^{65}$.
We have
$$\Ot(I) =  \Ot(I_{(p,\tau)}) = ( z^{26},\;  y, \; x)  \hbox{\quad for all primes\  } p \not\in S.$$
For the bad primes we obtain:\\
$\bullet$ \quad $\Ot(I_{(2,\sigma)}) = ( z^{17},\;  yz,\;   y^3, \;  xz^6, \;  xy^2, \;  x^2)$\\
$\bullet$ \quad $\Ot(I_{(7,\sigma)}) = ( z^{13},\;   y,\;   x^2)$\\
$\bullet$ \quad $\Ot(I_{(11,\sigma)}) = \Ot(I_{(55817,\sigma)}) = \Ot(I_{(p,\sigma)}) = ( z^{25},\;   yz,\;   y^2,\;   x)$
\end{example}

To conclude the paper we show an example illustrating
the merits of Corollary~\ref{cor:smallerisbad}.

\begin{example}\label{ex:badprimedetection}
Let $P = \QQ[x,y,z,w,s,t]$, let $\sigma$ be a term ordering
where each of $x, y, z, w$ is $\sigma$-greater than every
power product in $s$ and $t$,
and let $\tau$ be any elimination ordering
for $[s,t]$ which restricts to $\degrevlex$ on $\TT(x, y, z, w)$.
Let $f_1= t^3$, $f_2 = st^2 -2s^2$, $f_3 = s^2t -5$, $f_4 = s^3 -7t$,
and let $J = \ideal{x-f_1, \ y-f_2,\ z-f_3, \ w-f_4}$.
The given generators of $J$ form a reduced $\sigma$-\gbasis; so
clearly every prime is $\sigma$-good for $J$.
However, we do not know which ones are $\tau$-good.

Let us now look at $\Ot(I_{(p,\sigma)})$ for the primes $2,3,5$ and $7$.

\smallskip
For $p=2$ we compute $\Ot(I_{(p,\sigma)})$ obtaining the following tuple
{\footnotesize $$[y^2,\, z^5,\, yz^4,\, ty,\, sy,\, tx,\, sx,\, tz^3,\, sz^3,\, t^2,\, stz,\, s^2z,\, s^2t,\, s^3]$$}
\noindent
The best (and only!) tuple we have seen so far is $\Ot(I_{(2,\sigma)})$.

\medskip
For $p=3$ we compute $\Ot(I_{(p,\sigma)})$ to be the following
{\footnotesize
\begin{align*}
&[ z^5,\, yz^4,\, y^2z^3,\, y^3z^2,\, xy^2z^2,\, y^4z,\, xy^3z,\, y^5,\, xy^4,\, y^4w^2,\,
\BOX{xz^3w^3}\, ,\,  xz^4w^2,\, xyz^3w^2,\, x^2z^3w^2,\, x^2z^4w,\, \\
& x^2yz^3w,\, x^3z^3w,\, sz,\, sy,\, sx,\, tz^2,\, tyz,\, txz,\, ty^2,\, txy,\, tx^2,\, tw^3,\, sw^3,\, tzw^2,\, tyw^2,\, txw^2,\, t^2,\, st,\, s^2]
\end{align*}
}
If we compare $\Ot(I_{(3,\sigma)})$ with the best tuple, namely $\Ot(I_{(2,\sigma)})$,
we see from the very first elements of the tuples
that $\Ot( I_{(2,\,\sigma)}) \prec_\tau {\Ot(I_{(3,\,\sigma)})}$.
So the prime $2$ is surely $\tau$-bad.
And the best tuple we have seen so far is now $\Ot(I_{(3,\sigma)})$.

\smallskip
Next we choose $p = 5$.
For $\Ot(I_{(p,\sigma)})$ we get the following tuple
{\footnotesize
\begin{align*}
&[ z^5,\, yz^4,\, y^2z^3,\, y^3z^2,\, xy^2z^2,\, y^4z,\, xy^3z,\, y^5,\, xy^4,\, y^4w^2,\,
\BOX{y^2z^2w^3}\, ,\, xz^4w^2\, ,\, xyz^3w^2,\, x^2z^3w^2,\, x^2z^4w,\, \\
& x^2yz^3w,\, x^3z^3w,\, sz,\, sy,\, sx,\, tz^2,\, tyz,\, txz,\, ty^2,\, txy,\, tx^2,\, tw^3,\, sw^3,\, tzw^2,\, tyw^2,\, txw^2,\, t^2,\, st,\, s^2]
\end{align*}
}
If we compare $\Ot(I_{(5,\sigma)})$ with the best tuple seen so far
we see that the tuples agree up to the boxed elements, but
then we find that $\Ot( I_{(3,\,\sigma)}) \prec_\tau {\Ot(I_{(5,\,\sigma)})}$.
So the prime $3$ is surely $\tau$-bad, and we have a new best tuple,
namely $\Ot(I_{(5,\,\sigma)})$.

\smallskip
For the prime $p=7$ we compute the tuple
$\Ot(I_{(p,\sigma)})$ to be the following
{\footnotesize $$[z^3,\, y^2z^2,\, y^3z,\, y^4,\, sz,\, sy,\, sx,\, tw^2,\, sw^2,\, tzw,\, tz^2,\, tyz,\, ty^2,\, t^2w,\, stw,\, s^2w,\, t^3,\, st^2,\, s^2t,\, s^3]$$}
\noindent
If we compare $\Ot(I_{(7,\sigma)})$ with the best tuple seen so far,
just from comparing the very first elements, we find that $\Ot( I_{(7,\,\sigma)}) \prec_\tau {\Ot(I_{(5,\,\sigma)})}$.
So the prime $7$ is surely $\tau$-bad, and $\Ot(I_{(5,\,\sigma)})$ remains the best tuple.

\smallskip
If we try further primes, we find that they produce tuples equal to $\Ot(I_{(5,\,\sigma)})$.
At this point we are inclined to believe that $5$ is a good prime, but have no actual proof of this.

\smallskip
For this very small example, we can just
compute directly the reduced $\tau$-\gbasis
of~$J$; this will then confirm that $5$ is indeed good.
The main point is that once we have seen one good prime, the test from
Corollary~\ref{cor:smallerisbad} gives us a sure way of distinguishing
good primes from bad ones.  However, without some ``outside information''
we cannot know whether the best prime seen so far is actually good;
it may be just ``less bad'' than other primes tried.

\end{example}

\bibliographystyle{plain}

\end{document}